\documentclass{amsart}

\newcommand{\IZ}{\mathbb Z}
\newcommand{\IN}{\mathbb N}
\newcommand{\IQ}{\mathbb Q}
\newcommand{\IR}{\mathbb R}
\newcommand{\w}{\omega}

\newtheorem{theorem}{Theorem}

\newtheorem{proposition}{Proposition}
\newtheorem{corollary}{Corollary}
\newtheorem{lemma}{Lemma}
\newtheorem{problem}{Problem}
\theoremstyle{definition}
\newtheorem{definition}{Definition}
\newtheorem{remark}{Remark}

\begin{document}

\title{On 2-swelling topological groups}
\author{Taras Banakh}
\address{Ivan Franko National University of Lviv (Ukraine) and Jan Kochanowski University in Kielce (Poland)}
\email{t.o.banakh@gmail.com}

\begin{abstract}
A topological group $G$ is called {\em 2-swelling} if for any compact subsets $A,B\subset G$ and elements $a,b,c\in G$ the inclusions $aA\cup bB\subset A\cup B$ and $aA\cap bB\subset c(A\cap B)$ are equivalent to the equalities $aA\cup bB=A\cup B$ and $aA\cap bB=c(A\cap B)$.
We prove that an (abelian) topological group $G$ is 2-swelling if each 3-generated (resp. 2-generated) subgroup of $G$ is discrete. This implies that the additive group $\mathbb Q$ of rationals is 2-swelling and each locally finite topological group is 2-swelling.
\end{abstract}

\maketitle

In this paper we give a partial solution the following question \cite{Mur} of Alexey Muranov posted at MathOverFlow on July 24, 2013.

\begin{problem}[Muranov]\label{pr:Muranov} Let $A,B$ be compact subsets of a Hausdorff topological group $G$ such that $aA\cup bB\subset A\cup B$ and $aA\cap bB\subset c(A\cap B)$ for some elements $a,b,c\in G$. Is $aA\cup bB=A\cup B$ and $aA\cap bB=c(A\cap B)$?
\end{problem}

For example, for any real numbers $u<v<w<t$ the closed intervals $A=[u,w]$ and $B=[v,t]$ satisfy the equalities $(a+A)\cup (b+B)=A\cup B$ and $(a+A)\cap (b+B)=c+(A\cap B)$ for  $a=t-w$, $b=u-v$ and $c=v+w-u-t$.

First we observe that Muranov's Problem~\ref{pr:Muranov} has an affirmative answer for finite sets $A,B\subset G$.

\begin{proposition} Let $A,B$ be finite subsets of a group $G$ and $aA\cup bB\subset A\cup B$ and $aA\cap bB\subset c(A\cap B)$ for some elements $a,b,c\in G$. Then $aA\cup bB=A\cup B$ and $aA\cap bB=c(A\cap B)$.
\end{proposition}

\begin{proof} Evaluating the cardinality of the corresponding sets, we conclude that
$$
\begin{aligned}
|A\cup B|&\ge |aA\cup bB|=|aA|+|bB|-|aA\cap bB|\ge\\
&\ge |A|+|B|-|c(A\cap B)|=|A|+|B|-|A\cap B|=|A\cup B|,
\end{aligned}
$$
$|A\cup B|=|aA\cup bB|$ and $|aA\cap bB|=|c(A\cap B)|$, which imply the desired equalities $aA\cup bB=A\cup B$ and $aA\cap bB=c(A\cap B)$.
\end{proof}

The Muranov's problem can be reformulated as a problem of 2-parametric generalization of the classical  Swelling Lemma, see \cite[1.9]{CHK}.

\begin{theorem}[Swelling Lemma] For any compact subset $A$ of a Hausdorff topological semigroup $S$ and any element $a\in S$ the inclusion $aA\supset A$ is equivalent to the equality $aA=A$.
\end{theorem}

In this paper we shall give some partial affirmative answers to Muranov's Problem~\ref{pr:Muranov}. First, we introduce appropriate definitions.

\begin{definition} A topological group $G$ is called {\em 2-swelling} if for any compact subsets $A,B\subset G$ and elements $a,b,c\in G$ the inclusions $aA\cup bB\subset A\cup B$ and $aA\cap bB\subset c(A\cap B)$ are equivalent to the equalities $aA\cup bB=A\cup B$ and $aA\cap bB=c(A\cap B)$.
\end{definition}

\begin{definition} A topological group $G$ is called {\em weakly 2-swelling} if for any compact subsets $A,B\subset G$ and elements $a,b\in G$ with $aA\cap bB=\emptyset$ the inclusion $aA\cup bB\subset A\cup B$ implies $aA\cup bB=A\cup B$ and $A\cap B=\emptyset$.
\end{definition}

The following trivial proposition implies that the class of (weakly) 2-swelling topological groups is closed under taking subgroups.

\begin{proposition} Let $h:G\to H$ be a continuous injective homomorphism of topological groups. If the topological group $H$ is (weakly) 2-swelling, then so is the group $G$.
\end{proposition}

  A subgroup $H$ of a group $G$ is called {\em $n$-generated} for some $n\in\IN$ if $H$ is generated by set of cardinality $\le n$.
It can be shown that a topological group $G$ is Hausdorff if each 1-generated subgroup of $G$ is discrete.
The main result of this paper is the following theorem.

\begin{theorem}\label{t:main} An (abelian) topological group $G$ is
\begin{itemize}\itemsep=0pt\parskip=0pt
\item 2-swelling if each 3-generated (resp. 2-generated) subgroup of $G$ is discrete;
\item weakly 2-swelling if each 2-generated (resp. 1-generated) subgroup of $G$ is discrete.
\end{itemize}
\end{theorem}

\begin{corollary} For every $n\in\IN$ the group $\IQ^n$ is 2-swelling and the group $\IR^n$ is weakly 2-swelling.
\end{corollary}

\begin{corollary} Each locally finite topological group is 2-swelling.
\end{corollary}

We recall that a group $G$ is {\em locally finite} if each finite subset $F\subset G$ generates a finite subgroup. In the proof of Theorem~\ref{t:main} we shall use the following lemma.

\begin{lemma}\label{l1} Let $K\subset G$ be a compact subset of a topological group and $H$ be a discrete subgroup of $G$. Then the set $K\cap H$ is finite and $\sup_{x\in G}|K\cap Hx|\le |KK^{-1}\cap H|<\infty$. If the discrete subgroup $H$ is closed in $G$, then for every $n\in\IN$ the set $\{x\in G:|K\cap Hx|\ge n\}$ is closed in $G$.
\end{lemma}

\begin{proof} First we show that the intersection $K\cap H$ is finite. Assuming that $K\cap H$ is infinite, we can choose a sequence $(x_n)_{n\in\w}$ of pairwise distinct points of $K\cap H$. By the compactness of $K$ the sequence $(x_n)_{n\in\w}$ has an accumulation point $x_\infty\in K$.

Since the subgroup $H\subset G$ is discrete, the unit $1$ of $H$ has a neighborhood $U_1\subset G$ such that $U_1\cap H=\{1\}$. Choose a neighborhood $V_1\subset G$ of $1$ such that $V_1V_1^{-1}\subset U_1$.
Since $x_\infty$ is an accumulation point of the sequence $(x_n)_{n\in\w}$, the neighborhood $V_1x_\infty$ of $x_\infty$ contains two distinct points $x_n, x_m$ of the sequence. Then $x_nx_m^{-1}\in V_1V_1^{-1}\cap H\subset U_1\cap H=\{1\}$ and hence $x_n=x_m$, which contradicts the choice of the sequence $(x_k)$.
So, $K\cap H$ is finite. The set $KK^{-1}$ being a continuous image of the compact space $K\times K$ is compact too, which implies that $KK^{-1}\cap H$ is finite.

Next, we prove that $\sup_{x\in G}|K\cap Hx|\le |KK^{-1}\cap H|<\infty$. Given any point $x\in G$ with $K\cap Hx\ne\emptyset$, choose a point $y\in K\cap Hx$ and observe that $|K\cap Hx|=|K\cap Hy|=|Ky^{-1}\cap H|\le|KK^{-1}\cap H|$. Then $\sup_{x\in G}|K\cap Hx|\le|KK^{-1}\cap H|<\infty$.

Now assume that the discrete subgroup $H$ is closed in $G$. To see that for every $n\in\IN$ the set $G_n=\{x\in G:|K\cap Hx|\ge n\}$ is closed in $G$, choose any element $x\in G$ with $|K\cap Hx|<n$. Then the set $F=Kx^{-1}\cap H$ has cardinality $|F|=|K\cap Hx|<n$. It follows that $H\setminus F$ is a closed subset of $G$, disjoint with the compact set $Kx^{-1}$. For every $y\in K$ we can find a neighborhood $O_y\subset G$ of $y$ and a neighborhood $O_{x,y}\subset G$ of $x$ such that $O_yO_{x,y}^{-1}\cap (H\setminus F)=\emptyset$. By the compactness of $K$, the open cover $\{O_y:y\in K\}$ of $K$ has  finite subcover $\{O_y:y\in E\}$ (here $E\subset K$ is a suitable finite subset of $K$). Then the neighborhood $O_x=\bigcap_{y\in E}O_{x,y}$ of  $x$ has the property: $KO_x^{-1}\cap (H\setminus F)=\emptyset$, which equivalent to $KO_x^{-1}\cap H\subset F$ and implies $O_x\cap G_n=\emptyset$, witnessing that the set $G_n$ is closed in $G$.
\end{proof}

\begin{remark} It is well-known that any discrete subgroup $H$ of a Hausdorff topological group $G$ is closed in $G$. In general case this is not true: for any discrete topolgical group $H$ and an infinite group $G$ endowed with the anti-discrete topology, the subgroup $H\times\{1_G\}$ of $H\times G$ is discrete but not closed in $G$.
\end{remark}

The following theorem implies Theorem~\ref{t:main}, and is the main technical result of the paper.

\begin{theorem}\label{texnic} Let $A,B\subset G$ be two compact subsets of a topological group $G$ such that $aA\cup bB\subset A\cup B$ and $aA\cap bB\subset c(A\cap B)$ for some points $a,b,c\in G$.
The equalities $aA\cup bB=A\cup B$ and $aA\cap bB=c(A\cap B)$ hold if either the subgroup $H_3$ generated by the set $\{a,b,c\}$ is discrete or for some subset $T\subset\{a,b,c\}$ with $\{a,b\}\subset T$, $\{a,c\}\subset T$ or $\{b,c\}\subset T$ the subgroup $H_2$ generated by $T$ is discrete and closed in $G$, and $H_2$ is normal in the subgroup $H_3$.
\end{theorem}

\begin{proof} The proof splits into 2 parts.

I. The subgroup $H_3$ generated by the set $\{a,b,c\}$ is discrete. By the compactness of the sets $A,B,aA,bB$, for every $x\in G$ the sets $(aA\cup bB)\cap H_3x$ and
$(A\cup B)\cap H_3x$ are finite (see Lemma~\ref{l1}), so we can evaluate their cardinality:
$|(aA\cup bB)\cap H_3x|\le |(A\cup B)\cap H_3x|$ and $|aA\cap bB\cap H_3x|\le|c(A\cap B)\cap H_3x|=|A\cap B\cap c^{-1}H_3x|=|A\cap B\cap H_3x|$.
Next, observe that
$$
\begin{aligned}
|(A\cup B)\cap H_3x|&\ge |(aA\cup bB)\cap H_3x|=\\
&=|aA\cap H_3x|+|bB\cap H_3x|-|(aA\cap bB)\cap H_3x|\ge\\
& \ge |A\cap H_3x|+|B\cap H_3x|-|A\cap B\cap H_3x|=|(A\cup B)\cap H_3x|,
\end{aligned}
$$which implies that $|(A\cup B)\cap H_3x|=|(aA\cup bB)\cap H_3x|$ and $|(aA\cap bB)\cap H_3x|=|A\cap B\cap H_3x|$ and hence $(A\cup B)\cap H_3x=(aA\cup bB)\cap H_3x$ and $(aA\cap bB)\cap H_3x=A\cap B\cap H_3x$. Finally,
$$
aA\cup bB=\bigcup_{x\in G}(aA\cup bB)\cap H_3x=\bigcup_{x\in G}(A\cup B)\cap H_3x=A\cup B\mbox{ and}$$
$$aA\cap bB=\bigcup_{x\in G}(aA\cap bB)\cap H_3x=\bigcup_{x\in G}(A\cap B)\cap H_3x=A\cap B.
$$
\smallskip

II. The subgroup $H_3$ is not discrete but for some set $T\subset \{a,b,c\}$ with $\{a,b\}\subset T$, $\{a,c\}\subset T$ or $\{b,c\}\subset T$ the subgroup $H_2$ generated by $T$ is closed and discrete in $G$ and $H_2$ is normal in $H_3$. In this case the quotient group $H_3/H_2$ is not discrete. Let $t$ be the unique element of the set $\{a,b,c\}\setminus T$. It follows that $H_3=\bigcup_{n\in\IZ}H_2t^n$. The normality $H_2$ in $H_3$ implies that $H_2t=tH_2$. Depending on the value of $t$ three cases are possible.
\smallskip

(a) First consider the case of $t=a$ and $\{b,c\}\subset T$. In this case for every $x\in G$ the inclusions $aA\cap bB\cap H_2x\subset c(A\cap B)\cap H_2x$ and $(aA\cup bB)\cap H_2x\subset (A\cup B)\cap H_2x$ imply $|aA\cap bB\cap H_2x|\le|c(A\cap B)\cap H_2x|=|(A\cap B)\cap c^{-1}H_2x|=|(A\cap B)\cap H_2x|$ and
\begin{equation}\label{eq2}
\begin{aligned}
&|(A\cup B)\cap H_2x|\ge|(aA\cup bB)\cap H_2x|=\\
&=|aA\cap H_2x|+|bB\cap H_2x|-|(aA\cap bB)\cap H_2x|\ge\\
&\ge|A\cap a^{-1}H_2x|+|B\cap b^{-1}H_2x|-|(A\cap B)\cap H_2x|=\\
&=|A\cap a^{-1}\!H_2x|-|A\!\cap\! H_2x|+|A\cap H_2x|+|B\cap H_2x|-|(A\cap B)\cap H_2x|=\\
&=|A\cap a^{-1}H_2x|-|A\cap H_2x|+|(A\cup B)\cap H_2x|.
\end{aligned}
\end{equation}

Consequently,
\begin{equation}\label{x->a}
|A\cap H_2a^{-1}x|\le |A\cap H_2x|\mbox{ \ and \ }|A\cap H_2x|\le |A\cap H_2ax|\mbox{ \ for every $x\in G$}.
\end{equation}
By Lemma~\ref{l1}, for every $x\in G$ the number $\alpha_x=\max_{n\in\IZ}|A\cap H_2a^nx|$ is finite, so we can find a number $n_x\in\IZ$ such that $|A\cap H_2a^{n_x}x|=\alpha_x$.
By Lemma~\ref{l1}, the set $G_{\alpha}=\{g\in G:|A\cap H_2g|\ge\alpha_x\}$ is closed in $G$.

 The inequalities (\ref{x->a}) guarantee that $\alpha_x=|A\cap  H_2a^{n_x}x|\le|A\cap H_2a^{m}x|=\alpha_x$ for all $m\ge n_x$ and hence $\bigcup_{m\ge n_x}H_2a^mx\subset G_\alpha$. Taking into account that the quotient group $H_3/H_2$ is not discrete and is generated by the coset $H_2a$, we conclude that the set $\bigcup_{m\ge n_x}H_2a^mx$ is dense in $H_3$, and hence $H_3\subset G_\alpha$, which means that $|A\cap H_2a^mx|=\alpha_x$ for all $m\in\IZ$ and hence all inequalities in (\ref{x->a}) and (\ref{eq2}) turn into equalities. In particular, $|(aA\cup bB)\cap H_2x|=|(A\cup B)\cap H_2x|$ and $|(aA\cap bB)\cap H_2x|=|(A\cap B)\cap H_2x|=|c(A\cap B)\cap H_2x|$ for all $x\in G$. Combining these equalities with the inclusions $(aA\cup bB)\cap H_2x\subset (A\cup B)\cap H_2x$ and $(aA\cap bB)\cap H_2x\subset c(A\cap B)\cap H_2x$, we conclude that $(aA\cup bB)\cap H_2x= (A\cup B)\cap H_2x$ and $(aA\cap bB)\cap H_2x= c\,(A\cap B)\cap H_2x$ for all $x\in G$ and hence
 $$aA\cup bB=\bigcup_{x\in G}(aA\cup bB)\cap H_2x=\bigcup_{x\in G}(A\cup B)\cap H_2x=A\cup B\mbox{ and}$$
 $$aA\cap bB=\bigcup_{x\in G}(aA\cap bB)\cap H_2x=\bigcup_{x\in G}c(A\cap B)\cap H_2x=c(A\cap B).$$
\smallskip

(b) The case of $t=b$ and $\{a,c\}\subset T$ can be considered by analogy with the case (a).
\smallskip

(c) Finally we consider the case of $t=c$ and $\{a,b\}\subset H_2$.  Observe that for every $x\in G$ the inclusion $(aA\cap bB)\cap H_2x\subset c(A\cap B)\cap H_2x$ implies $$|aA\cap bB\cap H_2x|\le|c(A\cap B)\cap H_2x|=|A\cap B\cap c^{-1}H_2x|=|A\cap B\cap H_2c^{-1}x|.$$ On the other hand,
\begin{equation}\label{eq:c}
\begin{aligned}
&|A\cap B\cap H_2c^{-1}x|\ge |aA\cap bB\cap H_2x|=\\
&=|aA\cap H_2x|+|bB\cap H_2x|-|(aA\cup bB)\cap H_2x|\ge\\
&\ge |A\cap a^{-1}H_2x|+|B\cap b^{-1}H_2x|-|(A\cup B)\cap H_2x|=\\
&=|A\cap H_2x|+|B\cap H_2x|-(|A\cap H_2x|+|B\cap H_2x|-|A\cap B\cap H_2x|)=\\
&=|A\cap B\cap H_2x|.
\end{aligned}
\end{equation}
 Therefore,
\begin{equation}\label{x->cx}
|A\cap B\cap H_2x|\le |A\cap B\cap H_2c^{-1}x|.
\end{equation}

By Lemma~\ref{l1}, for every $x\in G$ the number $\Delta_x=\max_{n\in\IZ}|A\cap B\cap H_2c^nx|$ is finite, so we can find a number $n_x\in\IZ$ such that $|A\cap B\cap H_2c^{n_x}x|=\Delta_x$.
By Lemma~\ref{l1}, the set $G_{\Delta}=\{g\in G:|A\cap B\cap H_2g|\ge\Delta_x\}$ is closed in $G$.

 The inequality (\ref{x->cx}) guarantees that $\Delta_x=|A\cap B\cap c^{n_x}x|\le|A\cap B\cap H_2c^mx|\le\Delta_x$ for all $m\le n_x$ and hence $\bigcup_{m\le n_x}H_2c^mx\subset G_\Delta$. Taking into account that the quotient group $H_3/H_2$ is not discrete and is generated by the coset $H_2c$, we conclude that the set $\bigcup_{m\le n_x}H_2c^mx$ is dense in $H_3$, and hence $H_3\subset G_\Delta$. Then $|A\cap B\cap H_2c^mx|=\Delta_x$ for every $m\in\IZ$, which implies that all inequalities in (\ref{x->cx}) and (\ref{eq:c}) turn into equalities. In particular, we get the equalities
 $|A\cap B\cap H_2c^{-1}x|=|aA\cap bB\cap H_2x|$ and $|(aA\cup bB)\cap H_2x|=|(A\cup B)\cap H_2x|$, which imply the equalities $aA\cap bB\cap H_2x=c(A\cap B)\cap H_2x$ and $(aA\cup bB)\cap H_2x=(A\cup B)\cap H_2x$ holding for all $x\in G$. Then
 $$aA\cup bB=\bigcup_{x\in G}(aA\cup bB)\cap H_2x=\bigcup_{x\in G}(A\cup B)\cap H_2x=A\cup B\mbox{ \ and}$$
 $$aA\cap bB=\bigcup_{x\in G}(aA\cap bB)\cap H_2x=\bigcup_{x\in G}c(A\cap B)\cap H_2x=c(A\cap B).$$
\end{proof}

\begin{corollary} Let $A,B\subset G$ be two compact subsets of a topological group $G$ such that $aA\cup bB\subset A\cup B$ and $aA\cap bB=\emptyset$ for some points $a,b\in G$. If the cyclic subgroup $H_a$ generated by $a$ is closed and discrete in $G$ and $H_a$ is normal in the subgroup $H_{a,b}$ generated by the set $\{a,b\}$, then $aA\cup bB=A\cup B$ and $A\cap B=\emptyset$.
\end{corollary}

\begin{proof} Apply Theorem~\ref{texnic} with a point $c\in\{a,b,1_G\}$.
\end{proof}

The most intriguing problem about (weakly) 2-swelling groups concerns the real line (and the circle).

\begin{problem}[Muranov] Is the additive group $\IR$ of real numbers 2-swelling? Is the group $\IR/\IZ$ weakly 2-swelling?
\end{problem}

Theorem~\ref{texnic} implies that if $\IR$ is not 2-swelling, then it contains two compact sets $A,B\subset \IR$ such that $(a+A)\cup (b+B)\subset A\cup B$ and $(a+A)\cap (b+B)\subset c+(A\cap B)$ for some non-zero real numbers $a,b,c$ such that all fractions $\frac{a}{b}$, $\frac{a}{c}$, $\frac{b}{c}$ are irrational. The following proposition shows that such sets $A,B$ necessarily have positive Lebesgue measure.

\begin{proposition} Assume that $A,B\subset G$ are two non-empty compact subsets of $\IR$ such that $aA\cup bB\subset A\cup B$ for some non-zero real numbers $a,b$ with irrational fraction $\frac{a}{b}$. Then $A+b\IZ=\IR=B+a\IZ$ and hence the sets $A,B$ have positive Lebesgue measure.
\end{proposition}

\begin{proof} Given any point $x_0\in A$, for every $n\in\IN$ let $$x_{n}=\begin{cases}
a+x_{n-1}&\mbox{if $x_{n-1}\in A$},\\
b+x_{n-1}&\mbox{otherwise,}
\end{cases}
\mbox{ \ \ and \ \ }s_n=\begin{cases}
a&\mbox{if $x_{n-1}\in A$},\\
b&\mbox{otherwise}.
\end{cases}
$$
Using the inclusion $(a+A)\cup (b+B)\subset A\cup B$ we can prove by induction that
$\{x_n\}\subset A\cup B$. The compactness of $A\cup B$ implies that the sets $\{n\in\IN:s_n=a\}$ and $\{n\in\IN:s_n=b\}$ are infinite. Consider the cyclic subgroup $b\IZ$ generated by $b$ and let $q_b:\IR\to \IR/b\IZ$ be the quotient homomorphism. Observe that for every $n\in\IN$ $q_b(x_n)=q_b(x_n)$ if $s_n=b$ and $q_b(x_n)=q_b(a+x_{n-1})=q_b(a)+q_b(x_{n-1})$ if $x_n\in A$. Since the set $\{n\in\IN:s_n=a\}$ is infinite and the fraction $\frac{a}{b}$ is irrational, the set $q_b(\{x_n\}_{n\in\IN})$ is dense in the quotient group  $\IR/b\IZ$. The definition of the numbers $s_n$ guarantees that $q_b(\{x_n\})_{n\in\IN})\subset q_b(A)$. Now the compactness of $A$ guarantees that $q_b(A)=\IR/b\IZ$ and hence $A+b\IZ=\IR$.

By analogy, we can prove that $B+a\IZ=\IR$.
\end{proof}

\end{document}